\documentclass[12pt]{amsart}
\usepackage{times,fullpage}%,backref}
\usepackage{latexsym,amscd,amssymb}
\newtheorem{thm}{Theorem}
\newtheorem{prop}{Proposition}
\newtheorem{lem}{Lemma}
\newtheorem{cor}{Corollary}

\theoremstyle{remark}
\newtheorem{rem}{Remark}

\theoremstyle{definition}

\newcommand{\G}{\Gamma}

\newcommand{\C}{\mathbb{ C}}

\newcommand{\Z}{\mathbb{ Z}}

\DeclareMathOperator{\DEF}{def}

\newcommand{\ltb}{\beta_1}

\begin{document}

\title{The deficiencies of K\"ahler groups}
\author{D.~Kotschick}
\address{Mathematisches Institut, {\smaller LMU} M\"unchen,
Theresienstr.~39, 80333~M\"unchen, Germany}
\email{dieter@member.ams.org}
\date{May 17, 2012; \copyright{\ D.~Kotschick 2012}}
\thanks{Work done during a visit to the Institut Mittag-Leffler (Djursholm, Sweden). I am grateful to J.~Amor\'os,
G.~Kokarev and P.~Kropholler for useful comments.}
%57M05 fund. group, ...
%57M50 geometric structures in low dimensions
%14F35 fundamental group
%14F45 topological properties
%32J15 complex surfaces
%32J27 Kaehler manifolds: generalizations, classification
%32Q15 Kaehler manifolds
%20F05 generators, relations, and presentations
%20J05 Homological methods in group theory

\begin{abstract}
Generalizing the theorem of Green--Lazarsfeld and Gromov, we classify K\"ahler groups of deficiency at least two.
As a consequence we see that there are no K\"ahler groups of even and strictly positive deficiency. With the same 
arguments we prove that K\"ahler groups that are non-Abelian and are limit groups in the sense of Sela
are surface groups.
\end{abstract}

\maketitle

%\bigskip

\hfill {\small {\it It is only fair to say that our knowledge}}

\hfill {\small {\it  of deficiency is quite deficient.}\footnote{F.~R.~Beyl and J.~Tappe~\cite[p.~191]{BT}}

\bigskip

\section{Introduction}

The deficiency $\DEF (\Gamma)$ of a finitely presentable group $\Gamma$ is the maximum over all presentations 
of the difference of the number of generators and the number of relators. This invariant, which arises naturally in 
combinatorial group theory, is often difficult to compute. 
Its investigation has, almost from the beginning, had close connections to low-dimensional topology. For example, 
it is well-known that knot groups have deficiency one, and that the fundamental groups of closed three-manifolds 
have non-negative deficiency; cf.~Epstein~\cite{Epstein}. 

In K\"ahler geometry, the deficiency has also long been known to play a r\^ole because of the following:
%It has also been known for many years that the deficiency plays a r\^ole in K\"ahler geometry because of the following result:
\begin{thm}{\rm (Green--Lazarsfeld~\cite{GL}, Gromov~\cite{G})}\label{t:fewrel}
If the fundamental group of a compact K\"ahler manifold $X$ has deficiency $\geq 2$, then the Albanese image of $X$
is a curve.
\end{thm}
In fact, Gromov proved that the fundamental group is commensurable with a surface group, and so one might
hope for a more detailed classification. For example, it is natural to wonder whether the kernel of the homomorphism 
on fundamental groups induced by the Albanese is finite. However, in the more than twenty years since Theorem~\ref{t:fewrel} 
was proved, there have been no improvements on this result. The main purpose of this paper is to prove the following definitive 
version of Theorem~\ref{t:fewrel}:
\begin{thm}\label{t:master}
A group of deficiency $\geq 2$ is the fundamental group of a compact K\"ahler manifold $X$ if and
only if it is isomorphic to the orbifold fundamental group of a curve of genus $g\geq 2$. The curve is 
the Albanese image of $X$ and the isomorphism is induced by the Albanese map of $X$.
\end{thm}
The orbifold structure on the target is determined by the multiplicities of the singular fibers of the Albanese.
Since the kernel of the map from the orbifold fundamental group to the ordinary fundamental group is huge (not even
finitely generated~\cite{Gri}) whenever the orbifold structure is non-trivial, we see that the Albanese map usually does 
not induce a homomorphism with finite kernel on the ordinary fundamental group.
It is straightforward to see that all such orbifold fundamental groups are in fact fundamental groups of smooth complex projective 
varieties, and that the deficiency equals $2g-1$; cf.~Section~\ref{s:fiber} below. The latter statement gives the following result:
\begin{cor}\label{c:evendef}
There are no K\"ahler groups with positive even deficiency.
\end{cor}

Theorem~\ref{t:master} cannot be obtained from the 
arguments of Green--Lazarsfeld~\cite{GL} or Gromov~\cite{G}, and a new ingredient is required. Our proof, like 
Gromov's, uses $\ell^2$-cohomology, starting from the well known fact that groups of deficiency at least two have 
positive first $\ell^2$-Betti number. However, unlike in Gromov's approach, we do not use any $\ell^2$-Hodge theory 
or complex analysis, but instead shift the focus onto the formal properties of the first $\ell^2$-Betti number, replacing 
analytic techniques by algebraic ones. A crucial tool in our argument is the vanishing theorem for the first $\ell^2$-Betti 
numbers of groups admitting finitely generated infinite normal subgroups of infinite index due to L\"uck~\cite{lueckMA} 
and Gaboriau~\cite{gaboriauIHES}.

The question of classifying one-relator K\"ahler groups, which appeared in~\cite{A,Ara} not long after the works of 
Green--Lazarsfeld~\cite{GL} and Gromov~\cite{G}, is settled by the following result, to which a completely different route
was taken recently by Biswas and Mj~\cite{BM}.
\begin{thm}\label{t:main1}
An infinite one-relator group is the fundamental group of a compact K\"ahler manifold $X$
if and only if it is isomorphic to the orbifold fundamental group of an orbifold of genus $g\geq 1$ with at most one point
with multiplicity $>1$. Moreover, the isomorphism is induced by the Albanese map of $X$.
\end{thm}
If the number of generators of a one-relator group is at least $3$, then this follows from Theorem~\ref{t:master}.  
However, instead of deducing most of Theorem~\ref{t:main1} from Theorem~\ref{t:master}, we will give a direct argument 
for Theorem~\ref{t:main1} in Section~\ref{s:onerel} as a warmup for the proof of Theorem~\ref{t:master}. On the one hand, 
some of the details are actually simpler in this special case, and, on the other hand, this proof will also handle most 
one-relator groups with two generators, to which Theorem~\ref{t:master} does not apply because their deficiency is one. 

In Section~\ref{s:main} we state and prove a more detailed version of Theorem~\ref{t:master}, and we discuss 
a possible extension to groups of deficiency one. In Section~\ref{s:limit} we apply some of the arguments from the 
proof of Theorem~\ref{t:master} to show that a non-Abelian limit group in the sense of 
Sela~\cite{Sela} is a K\"ahler group if and only if it is a surface group of genus at least two.
In Section~\ref{s:relax} we show that our results are valid not only for K\"ahler manifolds, but also for non-K\"ahler
compact complex surfaces and for so-called Vaisman manifolds. The latter are a subclass of the locally conformally
K\"ahler manifolds, cf.~\cite{DO}. 

\section{Fibered K\"ahler manifolds and orbifolds}\label{s:fiber}

Suppose we have a surjective holomorphic map $f\colon X\longrightarrow C_g$ with connected fibers from a compact
complex manifold to a curve of genus $g\geq 1$. Then the induced map on fundamental groups is surjective. A subtle 
problem, which is easy to overlook, is 
that the kernel of this induced map may not be finitely generated. Indeed, if there are multiple fibers, it is not. For me, 
this point was clarified by Catanese's paper~\cite{CatF}, although it seems that Simpson~\cite{S} and perhaps others
understood it much earlier. In order to have a finitely generated kernel, one has to replace the usual fundamental 
group of $C_g$ by its orbifold fundamental group that takes into account the multiplicities of the multiple fibers.

\begin{lem}{\rm (\cite{CatF})}\label{l:C}
Let $f\colon X\longrightarrow C_g$ be a surjective holomorphic map with connected fibers from a compact
complex manifold to a curve of genus $g\geq 1$. By marking the critical values $p_1,\ldots,p_k$ of $f$ with suitable 
integral multiplicities $m_i\geq 1$, one can define the orbifold fundamental group $\pi_1^{orb}(C_g)$ of $C_g$
with respect to these multiplicities, so that one obtains a short exact sequence
\begin{equation}\label{eq:basic}
1\longrightarrow K\longrightarrow \pi_1(X)\longrightarrow \pi_1^{orb}(C_g)\longrightarrow 1
\end{equation}
in which the kernel $K$ is finitely generated, since it is a quotient of the fundamental group of a regular fiber of $f$.
\end{lem}

All these orbifold fundamental groups are also fundamental groups of compact K\"ahler manifolds, and even of smooth 
complex projective varieties. They are real surface groups only when the orbifold structure is trivial, but in all cases they 
can be realized by projective complex elliptic surfaces with appropriate multiple fibers with smooth reduction, 
see~\cite[Section~2.2.1]{FM}. 
%The fiber does not contribute to the fundamental group of the total space after introducing 
%suitable singular fibers without multiplicity. 
Alternatively, and less explicitly, one can use the fact that each such orbifold
fundamental group contains a real surface group as a subgroup of finite index~\cite{BN}, and then 
apply~\cite[Lemma~1.15]{ABCKT}.

The orbifold fundamental groups are of the form
\begin{equation}\label{eq:presgen}
\Gamma = \langle x_1,y_1,\ldots,x_g,y_g,z_1,\dots,z_k\ \vert \ [x_1,y_1]\ldots [x_g,y_g]\cdot z_1\ldots z_k=z_1^{m_1}=\ldots =z_k^{m_k}=1 \rangle
\end{equation}
for some $g\geq 1$ and $k\geq 0$, with $m_i\geq 2$ for all $i$. (Singular fibers with multiplicity $1$ do not 
contribute.)

\begin{lem}\label{l:def}
The deficiency of an orbifold group as in~\eqref{eq:presgen} is $\DEF (\Gamma)=2g-1$.
\end{lem}
\begin{proof}
The given presentation shows that the deficiency cannot be smaller than $2g-1$. It cannot be larger either because of the Morse 
inequality\footnote{This instance of the ``Morse inequality'' is often called ``Philip Hall's inequality'' by algebraists, because 
Epstein~\cite{Epstein} wrote it was ``essentially due'' to Hall.}
\begin{equation}\label{eq:MM}
\DEF (\Gamma)\leq b_1(\Gamma)-b_2(\Gamma) \ .
\end{equation}
in which the right-hand-side is $2g-1$ in this case.
\end{proof}

If the number $k$ of orbifold points is positive, then one can use the first relation in~\eqref{eq:presgen} to eliminate 
one of the generators $z_i$ from the presentation. Therefore, for $k\leq 1$ the orbifold fundamental group is a 
one-relator group of the form
\begin{equation}\label{eq:pres}
\Gamma = \langle x_1,y_1,\ldots,x_g,y_g \ \vert \ ([x_1,y_1]\cdot\ldots\cdot [x_g,y_g])^m =1 \rangle
\end{equation}
for some $g$ and $m$, both $\geq 1$. These are the groups appearing in Theorem~\ref{t:main1}.

\section{Background on $\ell^2$-Betti numbers}\label{s:ell}

The $\ell^2$-Betti numbers are defined for all countable groups. In this section we collect those results
about them that will be needed in the proofs of our theorems. For details we refer to L\"uck's monograph~\cite{lueckl2}.

For our purposes, the starting point is the relationship between the deficiency and the first $\ell^2$-Betti number $\ltb$.
\begin{prop}{\rm (\cite{hill})}\label{p:hill}
For every finitely presented group one has 
\begin{equation}\label{eq:Morse}
\ltb (\Gamma) \geq \DEF (\Gamma) -1 \ .
\end{equation}
In the case of equality the presentation complex of a presentation realizing the deficiency is aspherical.
\end{prop}
The inequality~\eqref{eq:Morse} has been part of the folklore for a very long time, and appears already in~\cite{G}.
It is a special case of the Morse inequalities in the $\ell^2$ setting. For the case of equality we refer to the
paper by Hillman~\cite{hill}.

Proposition~\ref{p:hill} shows that groups of deficiency $\geq 2$ have positive first $\ell^2$-Betti number. For one-relator groups
we have the following more precise statement, sharpening the predictions in~\cite{G,lueckl2}:
\begin{prop}{\rm (Dicks--Linnell~\cite{DL})}\label{t:DL}
Let $\Gamma$ be an infinite group given by $a\geq 2$ generators and one relation $r$. One may assume that $r$ is 
cyclically reduced and of the form $r=s^n$, with $n\geq 1$ and $s$ not a proper power. Then 
$$
\ltb (\Gamma) = -\chi (\Gamma) = a-1-\frac{1}{n} \ .
$$
\end{prop}
Here $\chi$ denotes the rational Euler characteristic, which is known to equal the $\ell^2$-Euler characteristic.
This proposition contains as special cases the computation of the first $\ell^2$-Betti numbers for surface groups and 
for the fundamental groups of two-dimensional orbifolds with exactly one point with multiplicity $>1$. We can easily extend 
to orbifolds with an arbitrary number of points with multiplicity, although this is not needed for the proofs of our main theorems.
\begin{prop}
Let $\pi_1^{orb}(C_g)$ be the fundamental group of an orbifold of genus $g\geq 1$ as in~\eqref{eq:presgen}. Then 
$$
\ltb (\pi_1^{orb}(C_g)) = -\chi (\pi_1^{orb}(C_g)) = 2g-2+k-\sum_{i=1}^{k}\frac{1}{m_i} \ .
$$
\end{prop}
\begin{proof}
By the solution of the Fenchel conjecture due to Bundgaard and Nielsen~\cite{BN}, $\pi_1^{orb}(C_g)$ has a surface 
group as a finite index subgroup. Since both $\ltb$ and $\chi$ are multiplicative under passage to finite index subgroups, 
the equality $\ltb = -\chi$ holds for $\pi_1^{orb}(C_g)$ because it holds for surface groups.
\end{proof}

The deepest theorem about $\ell^2$-Betti numbers that we use is the following vanishing result:
\begin{thm}{\rm (L\"uck~\cite{lueckMA}, Gaboriau~\cite{gaboriauIHES})}\label{t:LG}
If a finitely presentable $\Gamma$ fits into an exact sequence 
$$
1\longrightarrow K\longrightarrow \Gamma\longrightarrow Q\longrightarrow 1
$$
with $K$ and $Q$ infinite, and $K$ finitely generated, then $\ltb (\Gamma)=0$.
\end{thm}
This was originally proved by L\"uck~\cite{lueckMA} under the assumption that $Q$ is not a torsion group, and then
by Gaboriau~\cite{gaboriauIHES} in the general case. We apply Theorem~\ref{t:LG} to the exact sequence~\eqref{eq:basic}, where
L\"uck's technical assumption is satisfied, so that the full strength of Gaboriau's result is not required for our proofs.
A quick proof of L\"uck's result is given in~\cite{BMV}.

We will use Theorem~\ref{t:LG} to show that in certain situations the kernel $K$ in~\eqref{eq:basic} is finite. It is then
interesting to bound the order of $K$, which one can do with the following consequence of the Lyndon--Hochschild--Serre spectral 
sequence in the $\ell^2$ setting:
\begin{lem}\label{l:ss}
Suppose the group $\Gamma$ fits into an exact sequence 
$$
1\longrightarrow K\longrightarrow \Gamma\longrightarrow Q\longrightarrow 1
$$
with $K$ finite. Then $\ltb (\Gamma)=\frac{1}{\vert K\vert}\ltb (Q)$.
\end{lem}

\section{A direct approach to Theorem~\ref{t:main1}}\label{s:onerel}

As remarked in Section~\ref{s:fiber}, all the one-relator orbifold groups of the form~\eqref{eq:pres} are fundamental groups 
of smooth complex projective varieties. So one only has to prove the converse.

Let $\Gamma$ be an infinite one-relator group, and $X$ a compact K\"ahler manifold with $\pi_1(X)=\Gamma$.
Since the first Betti number of $\Gamma$ is positive, $X$ has a non-constant Albanese map $\alpha_X$.
The dimension of the Albanese image is bounded above by the cup-length of degree one cohomology
classes on $X$, see Catanese~\cite{Cat} and~\cite[Ch.~2]{ABCKT}. By the Hard Lefschetz theorem, this 
cup-length is at least two, and since $b_2(\Gamma)\leq 1$ because there is only one relation, we conclude that
the cup-length is exactly two. Therefore, the Albanese image of $X$ has real dimension equal to $2$. It is a 
complex curve $C_g$, necessarily of genus $g\geq 1$.
It is well known, and easy to see, that a one-dimensional Albanese image must be smooth, and that the fibers
of the Albanese map are connected in this case. Thus 
$$
\alpha_X\colon X\longrightarrow C_g
$$
is a surjective holomorphic map with connected fibers onto a smooth curve, which induces an isomorphism on 
$H^1(-;\Z)$ and an injection on $H^2(-;\Z)$. Since the target of the Albanese map is aspherical, $\alpha_X$ factors 
through the classifying map 
$$
c_X\colon X\longrightarrow B\Gamma
$$
of the universal covering of $X$, showing that the rank of $H^2(\Gamma;\Z)$ is positive. This means that the defining 
relation of $\Gamma$ is in the commutator subgroup of the free group on the generators. Thus there are 
precisely $2g$ generators.

Now we apply Lemma~\ref{l:C} to $\alpha_X$. Let $k$ be the number of critical values with multiplicity $>1$. 
If $k>0$, then the minimal number of generators of $\pi_1^{orb}(C_g)$ is $2g+k-1$, see~\cite{PRZ}. But 
$\pi_1^{orb}(C_g)$ is a homomorphic image of $\Gamma$, which is generated by $2g$ elements. Therefore 
we conclude that there is at most one critical value $p$ with a multiplicity $m>1$
%. Thus, $\pi_1^{orb}(C_g)$ has a presentation of the form given in~\eqref{eq:pres}, 
and the proof of Theorem~\ref{t:main1} will be complete once one shows that the kernel $K$ in Lemma~\ref{l:C} is trivial in this case.

By Proposition~\ref{t:DL} the first $\ell^2$-Betti number of a one-relator group $\Gamma$ on $2g$ generators is
\begin{equation}\label{eq:L2onerel}
\ltb (\Gamma) = 2g-1-\frac{1}{n} \ ,
\end{equation}
where $n$ is the maximal positive integer such that the defining relator can be written as an $n^{th}$ power.
Thus $\ltb (\Gamma)>0$ as soon as $g\geq 2$, or $g=1$ and $n>1$. We can then apply Theorem~\ref{t:LG}
to conclude that $K$ in~\eqref{eq:basic} is finite. Here it is crucial to have a finitely generated $K$.

We can estimate the order of the finite group $K$ using Lemma~\ref{l:ss} and~\eqref{eq:L2onerel}:
\begin{equation}\label{eq:estimate}
\vert K\vert = \frac{\ltb (\pi_1^{orb}(C_g))}{\ltb (\Gamma )} =\frac{2g-1-\frac{1}{m}}{2g-1-\frac{1}{n}}
< \frac{2g-1}{2g-1-\frac{1}{n}} = 1 + \frac{1}{(2g-1)n-1}\leq 2
\end{equation}
if $g\geq 2$, and if $g=1$ and $n>1$. Therefore, in all these cases, the proof is complete.

Finally, consider the case $g=n=1$. The condition $n=1$ implies that $\Gamma$ is torsion-free, so that $K$
would have to be infinite if it were non-trivial. The condition $g=1$ means that in this case the Albanese map 
%of any K\"ahler manifold with fundamental group $\Gamma$ 
is onto an elliptic curve. If there are no 
multiple fibers, then $\pi_1^{orb}(C_1)=\Z^2$, and the map $\Gamma\longrightarrow\Z^2$ induced by the 
Albanese is the Abelianization. Then the kernel $K$ is the commutator subgroup of $\Gamma$, and if it is 
finitely generated, then it is trivial by a result of Bieri~\cite[Corollary~B(b)]{Bieri2}.

\section{The main results}

\subsection{Proof of Theorem~\ref{t:master}}\label{s:main}

We now prove the following precise version of Theorem~\ref{t:master}.
\begin{thm}\label{t:Master}
A group $\Gamma$ of deficiency $\geq 2$ is the fundamental group of a compact K\"ahler manifold $X$ if and
only if it is isomorphic to $\pi_1^{orb}(C_g)$ for some $g\geq 2$. 

The curve $C_g$ is the Albanese image of $X$ and the isomorphism is induced by the Albanese map.
The number $k$ of multiple fibers of the Albanese map, equivalently the number of orbifold points with multiplicity $>1$
in $C_g$, is $0$ or $1$ if $\Gamma$ is a one-relator group, and equals the minimal number of relations in any 
finite presentation of $\Gamma$ in all other cases.
\end{thm}
\begin{proof}
We noted in Section~\ref{s:fiber} that all orbifold fundamental groups $\pi_1^{orb}(C_g)$ are in fact fundamental 
groups of smooth complex projective varieties, and that they have deficiency $\DEF (\pi_1^{orb}(C_g))=2g-1\geq 2$
if and only if $g\geq 2$. 

For the non-trivial direction of the theorem, let $X$ be a compact K\"ahler manifold with $\pi_1(X)=\Gamma$ of 
deficiency $\geq 2$.
The first step is to prove that the Albanese image is indeed a curve. Unlike for one-relator groups, this is not
obvious, and was the content of the original result of Green--Lazarsfeld and Gromov, cf.~Theorem~\ref{t:fewrel}.
For the sake of completeness we give a quick proof following Catanese~\cite{CatFew}. Using $\DEF (\Gamma)\geq 2$,
the Morse inequality~\eqref{eq:MM}
shows that every element of $H^1(\Gamma;\C)$ is contained in an isotropic subspace of dimension $\geq 2$ for the 
cup product $H^1(\Gamma;\C)\times H^1(\Gamma;\C)\longrightarrow H^2(\Gamma;\C)$. Therefore, by Catanese's proof
of the Siu--Beauville theorem, see~\cite{Cat} and~\cite[Ch.~2]{ABCKT}, every element of $H^1(\Gamma;\C)$ is 
in the image of a pullback $f^*$, where $f\colon X\longrightarrow C$ is a holomorphic map with connected fibers
to a curve of genus $\geq 2$. Since each of these fibrations is uniquely determined by the map it induces on integral 
cohomology, there are at most countably many of them. But $H^1(\Gamma;\C)$ cannot equal a union of 
countably many proper subspaces, and so at least one of these subspaces equals $H^1(\Gamma;\C)$. Therefore,
there is a holomorphic fibration of $X$ over a curve which induces an isomorphism on $H^1$. It follows by the 
universal property of the Albanese map that it factors through this fibration.

Since the Albanese map of $X$ has one-dimensional image, by Lemma~\ref{l:C} it induces an exact sequence
\begin{equation}\label{eq:KK}
1\longrightarrow K\longrightarrow \Gamma\longrightarrow \pi_1^{orb}(C_g)\longrightarrow 1 \ ,
\end{equation}
with $C$ of genus $g=\frac{1}{2}b_1(\Gamma)$, and $K$ finitely generated. 
By Proposition~\ref{p:hill}, $\ltb (\Gamma)\geq \DEF (\Gamma) -1\geq 1$, and so Theorem~\ref{t:LG} implies 
that $K$ is finite. It is this application of Theorem~\ref{t:LG} that requires the a priori knowledge that $K$ is 
finitely generated, which we arranged by considering the Albanese image as an orbifold.

If $k$ is the number of multiple fibers, then by~\cite{PRZ} the orbifold fundamental group $\pi_1^{orb}(C_g)$ is 
generated by no fewer than $2g+k-1$ elements. As $\Gamma$ surjects to $\pi_1^{orb}(C_g)$, we conclude that 
\begin{equation}\label{eq:kbd}
k\leq a-b_1(\Gamma)+1\leq b
\end{equation} 
if $\Gamma$ has a presentation with $a$ generators and $b$ relations. The second inequality follows from~\eqref{eq:MM}
since $b_2(\Gamma)\geq 1$. One can use~\eqref{eq:kbd} to bound the order of the finite group $K$, but, 
unlike in~\eqref{eq:estimate}, this bound is not good enough to show that $K$ is trivial, because
we do not have enough control over $\ltb (\Gamma )$. We therefore resort to an argument using group cohomology 
with mod $p$ coefficients to bound the deficiency of a group in order to rule out a non-trivial finite $K$ in~\eqref{eq:KK}.

%By assumption, $\Gamma$ has positive deficiency. 
Suppose that $\bar\Gamma\subset\Gamma$ is a subgroup of finite index $d$. If $\G$ has a presentation with $a$ generators
and $b$ relations, then, by the Reidemeister--Schreier process, $\bar\Gamma$ has a presentation with $\bar a = (a-1)d+1$
generators and $\bar b = bd$ relations. Therefore
$$
\DEF (\bar\Gamma )\geq \bar a-\bar b = (a-b-1)d+1 \ .
$$
This shows that the properties of having deficiency $\geq 2$, or of having positive deficiency, are preserved under passage 
to finite index subgroups. Therefore, the triviality of the kernel $K$ in~\eqref{eq:KK} follows from the following:
\begin{lem}{\rm (cf.~\cite[Thm.~1]{hill2})}\label{l:coho}
If the finite subgroup $K$ in~\eqref{eq:KK} is non-trivial, then $\Gamma$ has a finite index subgroup of negative deficiency.
\end{lem}
\begin{proof}
Suppose $K$ is non-trivial. Let $a\in K$ be an element of prime order $p$, and $C\subset K$ the cyclic subgroup generated by $a$.
The centralizer $\bar\Gamma = C_{\Gamma}(a)$ has finite index in $\Gamma$, and fits into a central extension
$$
1\longrightarrow C\longrightarrow \bar\Gamma\longrightarrow\Delta \longrightarrow 1 \ ,
$$
where $\Delta$ is a finite index subgroup of $\pi_1^{orb}(C_g)$. (For simplicity we may arrange that $\Delta$ is an ordinary 
surface group of positive genus~\cite{BN}.) This central extension corresponds to an Euler class $e\in H^2(\Delta;\Z_p)=\Z_p$. 
If we pull back to a subgroup $\bar\Delta\subset\Delta$ of index $p$, then $e$ is multiplied by $p$, and so becomes zero, so 
that the pulled-back extension splits. So $\bar\Gamma$ has a finite index subgroup isomorphic to $C\times\bar\Delta$, with 
$\bar\Delta$ a surface group. Now by the Morse inequality for cohomology with $\Z_p$-coefficients and the K\"unneth formula, 
we have 
$$
\DEF (C\times\bar\Delta )\leq \dim_{\Z_p} H^1(C\times\bar\Delta;\Z_p)-\dim_{\Z_p} H^2(C\times\bar\Delta;\Z_p)=-1 \ .
$$
This completes the proof of the Lemma.
\end{proof}
We have now proved that $\Gamma$ is isomorphic to $\pi_1^{orb}(C_g)$ via the homomorphism induced by the Albanese map.
Moreover, by~\eqref{eq:kbd}, the minimal number of relations in any finite presentation of $\Gamma$ is bounded from below by 
the number of multiple fibers of the Albanese map. If $\Gamma$ has a presentation with one relation, then there is at most one 
multiple fiber (but there may be none). For a two-relator group $\Gamma$ which is not isomorphic to a one-relator group we have 
at most two multiple fibers. If there were strictly fewer, then $\Gamma$ would in fact have a presentation with one relation only,
compare Section~\ref{s:fiber}. 
Therefore, there are exactly two multiple fibers in this case. Proceeding by induction, we see that for all larger values of the 
minimal number of relations in a presentation of $\Gamma$, this number equals the number of multiple fibers of the Albanese.
This completes the proof of Theorem~\ref{t:Master}.
\end{proof}

\begin{rem}
It would be interesting to extend Theorem~\ref{t:Master} to groups with $\DEF (\Gamma )=1$. 
The weaker result of Green--Lazarsfeld and Gromov, cf.~Theorem~\ref{t:fewrel}, does indeed extend:
\begin{prop}\label{p:pdefone}
If the fundamental group $\Gamma$ of a compact K\"ahler manifold $X$ has $\DEF (\Gamma )= 1$, then the Albanese image of $X$
is a curve.
\end{prop}
\begin{proof}
The assumption $\DEF (\Gamma )=1$ still gives $b_1(\Gamma )>0$, and so $X$ has a non-constant Albanese map.
In this case the Morse inequality~\eqref{eq:MM} gives 
\begin{equation}\label{eq:Mdefone}
b_2(\Gamma )\leq b_1(\Gamma )-1 \ .
\end{equation} 
If every element of $H^1(\Gamma;\C)$ is contained in an isotropic subspace of dimension $\geq 2$ for the 
cup product $H^1(\Gamma;\C)\times H^1(\Gamma;\C)\longrightarrow H^2(\Gamma;\C)$, then the conclusion
follows from Catanese's argument~\cite{CatFew} used in the proof of Theorem~\ref{t:Master} above.
If there is an element $\alpha\in H^1(\Gamma;\C)$ not contained in such an isotropic subspace 
%of dimension $\geq 2$ for the cup product $H^1(\Gamma,\R)\times H^1(\Gamma,\R)\longrightarrow H^2(\Gamma,\R)$, 
then the cup-product $\cup\alpha\colon H^1(\Gamma;\C)\longrightarrow H^2(\Gamma;\C)$ has only a one-dimensional
kernel. It is then surjective and~\eqref{eq:Mdefone} is an equality. But then every class in $H^2(\Gamma;\C)$ is a 
multiple of $\alpha$, and so the cup product $H^2(\Gamma;\C)\times H^2(\Gamma;\C)\longrightarrow H^4(\Gamma;\C)$ 
vanishes identically. Thus the cup-length of one-forms is less than four, and so the Albanese image is a curve after all, 
cf.~\cite{Cat} and~\cite[Ch.~2]{ABCKT}. 
\end{proof}
Thus we obtain the exact sequence~\eqref{eq:basic} with a finitely generated kernel $K$. Using $\DEF (\Gamma )>0$, 
we can rule out a non-trivial finite kernel $K$ as in the proof of Theorem~\ref{t:Master} 
using Lemma~\ref{l:coho}. If $\ltb (\Gamma )>0$, then Theorem~\ref{t:LG} rules out an infinite $K$, showing that
$\Gamma$ is isomorphic to the orbifold fundamental group of the Albanese image. 
However, it is now possible that $\ltb (\Gamma )=\DEF (\Gamma)-1=0$.
This is the case of equality in Proposition~\ref{p:hill}, and so $\Gamma$ would have an aspherical presentation 
complex. This shows that $\Gamma$ would be torsion-free of cohomological dimension two, and the subgroup
$K$ would be of cohomological dimension $\leq 2$. One can actually rule out a $K$ of cohomological dimension one,
so that in this case the desired generalization of Theorem~\ref{t:Master} follows.

Unfortunately, I do not know how to obtain the conclusion that $K$ is of cohomological dimension one in all cases.
It can be obtained from a result of Bieri~\cite[Thm.~5.5]{Bieri1} if one knows that $K$ is finitely presentable, or at least 
of type $FP_2$. Therefore, adding, for example, the assumption that the fundamental group is coherent, one can
circumvent this problem. There are other assumptions one can add, which also produce the desired conclusion,
compare the very recent preprint by Biswas and Mj~\cite{BM2}. However, I believe that none of these additional 
assumptions are necessary.

In all cases in which the kernel $K$ is trivial, the Albanese image in Proposition~\ref{p:pdefone} is an elliptic curve 
by Lemma~\ref{l:def}.
\end{rem}

\subsection{Limit groups}\label{s:limit}

The class of limit groups was introduced by Sela~\cite{Sela}, and has turned out to be the same as the class of fully 
residually free groups, see~\cite[Cor.~3.10]{CG}. A group $\Gamma$ is fully residually free if for every finite set 
$S\subset\Gamma$ there is a homomorphism $\varphi\colon\Gamma\longrightarrow F_n$ to a free group which is 
injective on $S$. As soon as $\Gamma$ is non-Abelian, this means that $\Gamma$ surjects onto a non-Abelian free group.

It was shown by Pichot~\cite{Pi} that non-Abelian limit groups $\Gamma$ have $\ltb (\Gamma)\geq 1$.
%, see also~\cite[Subsection~7.3]{PT}.
This allows us to use the same arguments as in the proof of Theorem~\ref{t:master} to conclude the following:
\begin{thm}\label{t:limit}
A non-Abelian limit group $\Gamma$ is a K\"ahler group if and only if it is isomorphic to a surface 
group of genus $\geq 2$.
\end{thm}
\begin{proof}
A non-Abelian limit group $\Gamma$ admits a surjective homomorphism to a non-Abelian free group.
Therefore, if such a $\Gamma$ is the fundamental group of a compact K\"ahler manifold $X$, then by the 
Siu--Beauville theorem, cf.~\cite[Ch.~2]{ABCKT}, $X$ admits a surjective holomorphic map $f\colon X\longrightarrow C_g$
with connected fibers to a curve of genus $g\geq 2$. We now apply Lemma~\ref{l:C} to $f$. Since $\Gamma$ has positive 
first $\ell^2$-Betti number~\cite{Pi}, the kernel $K$ must be finite by Theorem~\ref{t:LG}. However, limit groups are 
torsion-free~\cite[Prop.~3.1]{CG}, and so the finite kernel $K$ is actually trivial. Thus $\Gamma = \pi_1^{orb}(C_g)$, 
and the torsion-freeness implies that there are no points with multiplicity $>1$ in $C_g$, equivalently $f$ has no multiple 
fibers. Thus $\Gamma$ is an ordinary surface group.

Conversely, surface groups are indeed limit groups, cf.~\cite[p.~27]{CG}.
\end{proof}

%\newpage 

\section{The possible deficiencies of K\"ahler groups}

We have seen in Corollary~\ref{c:evendef} that there are no K\"ahler groups of positive even deficiency. This raises the 
question whether there are other constraints on the deficiencies of K\"ahler groups. In the case of positive deficiency
there are no further constraints, since all odd positive integers
are indeed deficiencies of K\"ahler groups because they are deficiencies of surface groups by Lemma~\ref{l:def}.

For negative deficiencies we have the following result.

\begin{prop}
Let $\Sigma_g$ be a closed oriented surface of genus $g\geq 1$, and $p>0$ a prime number. We have
\begin{align*}
\DEF (\pi_1(\Sigma_g)\times\Z^2) &=-2g \ ,\\
\DEF (\pi_1(\Sigma_g)\times \Z_p) &=-1 \ ,\\ 
\DEF (\pi_1(\Sigma_g)\times\Z^4) &=-3-6g \ , \\
\DEF (\pi_1(\Sigma_g)\times\Z^2\times (\Z_p)^2) &=-5-6g \ , \\
\DEF (\pi_1(\Sigma_g)\times (\Z_p)^4) &=-7-6g \ . 
\end{align*}
Furthermore, all these groups are K\"ahler.
\end{prop}
\begin{proof}
It is clear that all the groups are K\"ahler, because they are direct products of fundamental groups of curves and of finite groups. 
All finite groups are K\"ahler by a general result of Serre, but for $\Z_p$ one does not need this general result.

The calculation of the deficiency proceeds in the same way for all the examples. On the one hand, there is an obvious presentation
in each case, which shows that the deficiency cannot be smaller than the value claimed. On the other hand,
the deficiency cannot be larger either, because of the Morse inequality~\eqref{eq:MM}. In the cases where finite factors
appear, the Morse inequality for cohomology with $\Z_p$ coefficients is used.
\end{proof}
The proposition shows that almost all negative integers are deficiencies of K\"ahler groups. The only values not covered 
are $-3$, $-5$ and $-7$. Of these, the first is taken care of by $\DEF ((\Z_p)^3)=-3$. Finally, there are many K\"ahler groups
of deficiency zero, so that $-5$ and $-7$ are the only integers for which we have not decided the question whether they are
deficiencies of K\"ahler groups. It is very likely that these two values can be realized, for example by finite groups. 

\section{Relaxing the K\"ahler condition}\label{s:relax}

It was proved by Taubes~\cite{T} that every finitely presentable group is the fundamental group of a compact
complex three-fold. Therefore, it is interesting to look at the fundamental groups of complex manifolds that 
are not necessarily K\"ahler, but have Hermitian metrics satisfying geometrically interesting conditions weaker 
than the K\"ahler condition, in order to see where the strong restrictions on K\"ahler groups break down. 

Kokarev~\cite{Kok} introduced the class of so-called pluri-K\"ahler--Weyl manifolds, and in~\cite{KK} it was proved 
that non-Abelian free groups cannot be fundamental groups of such manifolds. We now generalize this result
by looking at arbitrary groups with deficiency $\geq 2$.

It was shown by Ornea--Verbitsky~\cite{OV} that non-K\"ahler pluri-K\"ahler--Weyl manifolds of complex dimension 
$\geq 3$ also have non-K\"ahler Vaisman structures, meaning that there is a complex structure with a Hermitian metric 
whose fundamental two-form satisfies $d\omega  = \omega\wedge\theta$ for a parallel non-zero one-form $\theta$,
cf.~\cite{DO}.
Therefore~\cite{Kok,KK}, the class of pluri-K\"ahler--Weyl manifolds in the sense of~\cite{Kok} consists of the union of the 
following three classes: K\"ahler manifolds, arbitrary compact complex surfaces, and manifolds that are, up
to diffeomorphism, non-K\"ahler Vaisman of complex dimension $\geq 3$. We now show that for these 
manifolds there are no additional deficiency $\geq 2$ fundamental groups other than those
occurring in Theorem~\ref{t:master}.

\begin{lem}\label{l:VII}
For surfaces $S$  with $b_1(S)=1$ one has $\DEF (\pi_1(S))\leq 1$. This bound is sharp.
\end{lem}
\begin{proof}
The implication is clear, cf.~\eqref{eq:MM}.
The bound is sharp because $\Z$ occurs as the fundamental group of certain Hopf surfaces of class VII.
\end{proof}

\begin{prop}\label{l:elliptic}
Let $S$ be a non-K\"ahler elliptic surface with $b_1(S)\geq 3$. Then $\ltb (\pi_1(S))=0$ and $\DEF (\pi_1(S))\leq 0$.
\end{prop}
\begin{proof}
The elliptic fibration induces an exact sequence of the form
$$
1\longrightarrow \Z^{2}\longrightarrow \pi_1(S)\longrightarrow \pi_1^{orb}(C_g)\longrightarrow 1 \ ,
$$
with $g\geq 1$; cf.~\cite[Section~2.7.2]{FM}. Therefore, by Theorem~\ref{t:LG}, one has $\ltb (\pi_1(S))=0$, 
implying $\DEF (\pi_1(S))\leq 1$ via~\eqref{eq:Morse}. In the case of equality in~\eqref{eq:Morse}, the group $\pi_1(S)$ 
would have to be of cohomological dimension at most $2$ by Proposition~\ref{p:hill}, which is 
not possible. In fact, the cohomological dimension is $4$, see~\cite[Cor.~1.41]{ABCKT}.
\end{proof}

\begin{prop}\label{p:Vaisman}
For a non-K\"ahler Vaisman manifold $V$ one has $\ltb (\pi_1(V))=0$ and $\DEF (\pi_1(V))\leq 1$. This bound is sharp.
\end{prop}
\begin{proof}
By definition, a non-K\"ahler Vaisman manifold $V$ carries a parallel non-zero one-form $\theta$, and therefore is the total space
of a smooth fiber bundle over the circle. Thus, its fundamental group fits into an extension of the form
$$
1\longrightarrow \pi_1(F)\longrightarrow \pi_1(V)\longrightarrow \Z\longrightarrow 1 \ ,
$$
with $F$ the fiber of the fibration over $S^1$. If $\pi_1(F)$ is finite, then $\ltb (\pi_1(V))=0$ follows
from $\ltb (\Z)=0$ by Lemma~\ref{l:ss}. If $\pi_1(F)$ is infinite, then $\ltb (\pi_1(V))=0$ follows
from Theorem~\ref{t:LG}. Now the vanishing of $\ltb (\pi_1(V))$ implies $\DEF (\pi_1(V))\leq 1$ by~\eqref{eq:Morse}. 
This bound is sharp because $\Z$ occurs as the fundamental group of certain Vaisman manifolds of Hopf type.
\end{proof}

Putting together these observations, we see that Theorems~\ref{t:master} and~\ref{t:limit} extend in the 
following way:
\begin{cor}
If $\Gamma$ is either a group with $\DEF (\Gamma)\geq 2$ or a non-Abelian limit group and $\Gamma$ is 
the fundamental group of a compact pluri-K\"ahler--Weyl manifold $X$, then $X$ is K\"ahler.
%In particular, Theorem~\ref{t:master}, respectively Theorem~\ref{t:limit}, applies.
\end{cor}
\begin{proof}
A compact complex surface is K\"ahler if and only if its first Betti number is even~\cite{BPV}. Thus we 
have to consider surfaces with odd $b_1$. By the Kodaira classification surfaces $S$ with odd first Betti number 
$\geq 3$ are elliptic, see~\cite{ABCKT,BPV}. 

The vanishing of their first $\ell^2$-Betti number shows that fundamental groups of elliptic surfaces with $b_1\geq 3$
or of non-K\"ahler Vaisman manifolds cannot have deficiency $\geq 2$, or be non-Abelian limit groups.

Groups with $b_1=1$ have deficiency at most one, and they cannot surject to non-Abelian free groups. 
\end{proof}

Finally, we note that Proposition~\ref{p:Vaisman} has the following consequence:
\begin{cor}
The fundamental group of a Vaisman manifold has finitely many ends.
\end{cor}
\begin{proof}
As is now well-known, groups with infinitely many ends have positive first $\ell^2$-Betti number, cf.~\cite{ABCKT,ABR}.
Thus a Vaisman manifold with a fundamental group with infinitely many ends would have to be K\"ahler by
Proposition~\ref{p:Vaisman}. However, in the K\"ahler case the number of ends is at most one by Gromov's 
result~\cite{G}; cf.~also~\cite{ABCKT,ABR}.
\end{proof}

\bibliographystyle{amsplain}

\bigskip

\end{document}